\documentclass{amsart} 

\usepackage{hyperref}
\usepackage[capitalize]{cleveref}
\usepackage{amssymb}

\numberwithin{equation}{section}

\newtheorem{thm}[equation]{Theorem}
\newtheorem{prop}[equation]{Proposition}
\newtheorem{cor}[equation]{Corollary}
\newtheorem{lem}[equation]{Lemma}

\theoremstyle{definition}
\newtheorem{defn}[equation]{Definition}
\newtheorem{defns}[equation]{Definitions}
\newtheorem{notation}[equation]{Notation}
\newtheorem{eg}[equation]{Example}
\newtheorem{rem}[equation]{Remark}
\newtheorem{rems}[equation]{Remarks}
\newtheorem*{ack}{Acknowledgments}

\crefformat{rems}{Remark~#2#1#3}
\crefmultiformat{thm}{Theorems~#2#1#3}{ and~#2#1#3}{, #2#1#3}{, and~#2#1#3}

\makeatletter
\def\subsection{\def\@secnumfont{\bfseries}%
\@startsection{subsection}{2}%
  {\z@ {\normalfont\bfseries\S}}{.5\linespacing\@plus.7\linespacing}{-.5em}%
  {\normalfont\bfseries}}
\makeatother

\newcommand{\pref}[1]{(\ref{#1})}
\newcommand{\fullcref}[2]{\cref{#1}\pref{#1-#2}}
\renewcommand{\Cref}{\cref}

\newcommand{\AND}{\mathbin{\mathrm{and}}}
\newcommand{\OR}{\mathbin{\mathrm{or}}}
\newcommand{\rel}{\mathrel{R}}
\newcommand{\rels}{\mathfrak{R}}
\newcommand{\relH}{r}
\newcommand{\closure}{\overline}

\DeclareMathOperator{\LO}{LO}

\newcommand{\N}{\mathbb{N}}
\newcommand{\Z}{\mathbb{Z}}

\newcommand{\R}{\mathbb{R}}

\makeatletter
\newcommand{\noprelistbreak}{\smallskip\@nobreaktrue\nopagebreak} 
\makeatother

\begin{document}

\title[Amenable groups with a locally invariant order]
{Amenable groups with a locally invariant order are locally indicable}

\author{Peter Linnell}

\address
{Department of Mathematics, Virginia Tech,
Blacksburg, Virginia \text{24061--0123}, USA}
\email
{plinnell@math.vt.edu,
\href{http://www.math.vt.edu/people/plinnell/}{http://www.math.vt.edu/people/plinnell/}}
\thanks{The first author was partially supported by a grant from the
NSA}

\author{Dave Witte Morris}

\address
{Department of Mathematics and Computer Science,
University of Lethbridge, Lethbridge, Alberta, T1K~3M4, Canada}
\email{
{Dave.Morris@uleth.ca}, 
\href{http://people.uleth.ca/~dave.morris/}{http://people.uleth.ca/\!$\sim$dave.morris/}}

\begin{abstract}
We show that every amenable group with a locally invariant partial order has a left-invariant total order (and is therefore locally indicable).
We also show that if a group~$G$ admits a left-invariant total order, and $H$ is a locally nilpotent subgroup of~$G$, then a left-invariant total order on~$G$ can be chosen so that its restriction to~$H$ is both left-invariant and right-invariant.
Both results follow from recurrence properties of the action of~$G$ on its binary relations.
\end{abstract}



\date{\today}

\maketitle

\section{Introduction} \label{IntroSect}

The purpose of this note is to point out two easy consequences of the proof that finitely generated, amenable, left-orderable groups have nontrivial first Betti number \cite{Morris-AmenOnLine}. 
(See \cref{DefnsSubsect} for the relevant definitions.)

Any left-invariant total order is a locally invariant order, so it is obvious that every left-orderable group has a locally invariant order. There is no known counterexample to the converse \cite[p.~1163]{Chiswell-LIO}, and we show that the converse is indeed true for amenable groups. (In particular, the converse is true for all virtually solvable groups. This does not seem to be trivial even for groups that are virtually abelian.)

\begin{thm} \label{AMENLIO}
Every amenable group with a locally invariant order is left-orderable. Therefore, the group is locally indicable.
\end{thm}

We also prove a new result on extending an ordering of a subgroup to an ordering of the ambient group:

\begin{thm} \label{RESTONILPISBI}
If
\noprelistbreak
	\begin{itemize} \itemsep=\smallskipamount
	\item $G$ is a left-orderable group,
	and
	\item $H$ is a locally nilpotent subgroup of~$G$,
	\end{itemize}
then there is a left-invariant total order on~$G$, such that the restriction of the order to~$H$ is bi-invariant.
\end{thm}

\begin{rem}
The subgroup~$H$ is not assumed to be convex, or normal (or anything else, other than locally nilpotent), so it is difficult to imagine how \cref{RESTONILPISBI} could be attacked by the classical methods of the theory of orderable groups. However, we will see that it (and also \cref{AMENLIO}) can be proved very easily by using the action of~$G$ on the space of its left-invariant orders, an idea that was recently introduced into the subject by \'E.\,Ghys and A.\,S.\,Sikora.
See \cite{Navas-DynamicsOfLO} for more discussion and applications of this method.
\end{rem}

Here is an outline of the paper. 
	\cref{PrelimSect} provides some standard definitions and discusses the topology on the space of binary relations. 
	\cref{RecurSect} explains the use of amenability to obtain recurrence in the space of binary relations.
	\cref{AMENLIOPfSect} proves \cref{AMENLIO}.
	\cref{RESTONILPISBIPfSect} proves \cref{RESTONILPISBI}.
	Finally, \cref{DiffuseSect} shows that groups with a locally invariant order can also be characterized as the groups that are ``diffuse'' or ``weakly diffuse'' in the sense of B.\,Bowditch \cite{Bowditch-VarUPP}.

\begin{ack}
We thank A.\,Navas, A.\,Rhemtulla, and the other participants in the workshop on ``Ordered Groups and Topology'' (Banff International Research Station, Alberta, Canada, February 12--17, 2012) for many helpful conversations, for providing the impetus for this research, and for pointing out that the word ``locally'' could be inserted into the statement of \cref{RESTONILPISBI}. \fullcref{Generalize}{NotForResid} was provided by D.\,Rolfsen (an organizer of the workshop). 
We also thank the BIRS staff for the warm hospitality that provided such a stimulating research environment, and A.\,M.\,W.\,Glass for helpful comments on a previous version of this paper.
\end{ack}

\section{Preliminaries} \label{PrelimSect}

\subsection{Some standard definitions} \label{DefnsSubsect}

\begin{defns}[{\cite{KopytovMedvedev-ROGrps}}]
Let $G$ be a group.
\noprelistbreak
	\begin{itemize} \itemsep=\smallskipamount
	\item A \emph{partial order} on~$G$ is a transitive, irreflexive binary relation $\prec$ on~$G$. That is, $x \not\prec x$, and, for all $x, y, z \in G$, if $x \prec y$ and $y \prec z$, then $x \prec z$.
	\item A \emph{total} (or ``linear'') order on~$G$ is a partial order~$\prec$, such that, for all $x, y \in G$ with $x \neq y$, we have either $x \prec y$ or $x \succ y$.
	\item $\prec$ is \emph{left-invariant} if, for all $x,y,g \in G$, we have $x \prec y \Rightarrow gx \prec gy$.
	\item $\prec$ is \emph{bi-invariant} if it is both left-invariant and right-invariant. That is, if $x \prec y$, then $gx \prec gy$ and $xg \prec yg$, for all $x,y,g \in G$.
	\item $G$ is \emph{left-orderable} if there exists a left-invariant total order on~$G$.
	\item $G$ is \emph{locally nilpotent} if every finitely generated subgroup of~$G$ is nilpotent.
	\item $G$ is \emph{locally indicable} if every nontrivial finitely generated subgroup of~$G$ has an infinite, cyclic quotient.
	\end{itemize}
\end{defns}

\begin{defn}[{\cite{Chiswell-LIO}}]
A partial order $\prec$ on~$G$ is \emph{locally invariant} if, for all $x,y \in G$ with $y \neq e$, we have either $xy \succ x$ or $xy^{-1} \succ x$.
\end{defn}

\begin{rem}
It is an easy exercise \cite[Lem.~1.1]{Chiswell-LIO} to show that a group $G$ has a locally invariant order iff there exists a partially ordered set $(\mathcal{P}, \prec)$ and a function $\rho \colon G \to \mathcal{P}$, such that, for all $x,y \in G$ with $y \neq e$, we have either $\rho(xy) \succ \rho(x)$ or $\rho(xy^{-1}) \succ \rho(x)$. (When $G$ is countable, one may take $(\mathcal{P},\prec)$ to be $(\R, <)$.) For example, $\R^n$ has a locally invariant order, because we may take $\rho(x) = \|x\|$.
\end{rem}

The notion of an amenable group has many different definitions that are all equivalent to one another. We choose the one that is most convenient for our purposes.

\begin{defn}[{\cite[p.~9 and Thm.~5.4(i,iii)]{Pier}}] \label{AmenableDefn} \ 
\noprelistbreak
 \begin{itemize} \itemsep=\smallskipamount
\item A measure~$\mu$ on a measure space~$X$ is said to be a
\emph{probability measure} if $\mu(X) = 1$.
\item A (discrete) group~$G$ is \emph{amenable} if for every continuous
action of~$G$ on a compact, Hausdorff space~$X$, there is a
$G$-invariant probability measure on~$X$.
\end{itemize}
 \end{defn}

 \begin{eg}[{\cite[Cors.~13.5 and 13.10]{Pier}}] \label{LocNilpIsAmen} 
 It is fairly easy to see that every solvable group is amenable. It is also easy to see that if every finitely generated subgroup of~$G$ is amenable, then $G$ is amenable. Therefore, every locally solvable group is amenable. In particular, every locally nilpotent group is amenable.
 \end{eg}

We also need the following two facts. The second is an easy observation, but the first is nontrivial.

\begin{lem}[{}{\cite[Props.~13.3 and 13.4]{Pier}}] \label{AmenFacts}
Assume $G$ is amenable. Then:
	\begin{enumerate}
	\item \label{AmenFacts-subgrp}
	every subgroup of~$G$ is amenable,
	and
	\item \label{AmenFacts-GxG}
	$G \times G$ is amenable.
	\end{enumerate}
\end{lem}

\subsection{Topology and action on the space of binary relations} \label{TopologySubsect}
A.\,S.\,Sikora \cite{Sikora-TopLO} introduced a topology on the space of left-invariant total orders on~$G$, and \'E.\,Ghys (personal communication) observed that it would be useful to study the natural action of~$G$ on this space. For our present purposes, we describe these ideas in the context of more general binary relations on~$G$, not just left-invariant orders.

\begin{defn} \label{P(X)Topology}
The collection of all subsets of a set~$X$ can be identified with the collection $2^X$ of all functions $f \colon X \to \{0,1\}$ (by identifying a subset with its characteristic function). Since $2^X$ can also be viewed as the Cartesian product of $\#X$ copies of the finite set $\{0,1\}$, Tychonoff's Theorem provides it with a natural topology, in which it is a compact Hausdorff space. (And it is metrizable if $X$ is countable.)
\end{defn}

\begin{defn}
For any set~$X$, each subset of $X \times X$ is said to be a \emph{binary relation} on~$X$. Therefore, \cref{P(X)Topology} tells us that the set of all binary relations on~$X$ has the topology of a compact Hausdorff space. (Hence, the same is true for any of its closed subsets.)
The topology is defined so that
	$$ \text{for any $x,y \in X$, the subset $\{\, R \in 2^{X \times X} \mid x \rel y \,\}$ is both open and closed} . $$
Therefore, any subset that is defined by a Boolean combination of finitely many assertions of the form $x_1 \prec y_1$, $x_2 \prec y_2$, \dots,  $x_n \prec y_n$ is also closed (and open). So the intersection of any collection of such subsets (even an infinite collection) is closed (but may not be open).
\end{defn}

\begin{rem}
 For any subgroup~$H$ of~$G$, there is a natural restriction map from $2^{G \times G}$ to $2^{H \times H}$. It is obvious that this is continuous.
\end{rem}

\begin{defn}
Let $G$ be an abstract group. Then $G$ acts on $2^{G \times G}$ by both left-translations and right-translations. These commute, so there is an action of $G \times G$ on $2^{G \times G}$, defined by
	$$ x \rel^{(g,h)} y \iff g x h^{-1} \rel g y h^{-1} .$$
It is clear that this is an action by homeomorphisms.
\end{defn}

\begin{eg}
Let $G$ be a group. Here are some important examples of closed subsets of~$2^{G \times G}$ that are invariant under the action of $G \times G$.
	\begin{enumerate}
	\item The set of all partial orders on~$G$, defined by the axioms
		$$ x \not\prec x $$
		$$ (x \prec y) \AND (y \prec z) \implies x \prec z $$
	\item The set of locally invariant orders on $G$, defined by the axioms for a partial order, together with
		$$ y \neq e \implies (x \prec xy) \OR (x \prec xy^{-1}) $$
	\item The set of all total orders on~$G$, defined by the axioms for a partial order, together with
		$$ x \neq y \implies (x \prec y) \OR (x \succ y) $$
	\item (Sikora \cite{Sikora-TopLO})
	The set of left-invariant total orders on $G$, defined by the axioms for a total order, together with
		$$ x \prec y \implies zx \prec zy $$
	\item (Navas \cite[Prop.~3.7]{Navas-DynamicsOfLO})
	The set of Conradian orders on~$G$, defined by the axioms for a left-invariant total order, together with
		$$ (x \succ e) \AND (y \succ e)  \implies xy^2 \succ y $$
	\end{enumerate}
\end{eg}

\section{Recurrence in the space of binary relations} \label{RecurSect}

\begin{defn}[{}{cf.\ \cite[Defn.~3.2]{Morris-AmenOnLine}}]
Let $G$ be a group, and let ${\rel} \in 2^{G \times G}$. 
	\begin{itemize}
	\item For $(g,h) \in G \times G$, we say $\rel$ is \emph{recurrent for $(g,h)$} if, for every finite subset $F$ of~$G$, there exists $n \in \Z^+$, such that $\rel^{(g,h)^n}$ and~$\rel$ have the same restriction to~$F$. (If $G$ is countable, this is equivalent to the assertion that there is a sequence $n_i \to \infty$, such that $\rel^{(g,h)^{n_i}} \to \rel$ as $i \to \infty$.)
	\item $\rel$ is \emph{recurrent} if it is recurrent for every element of $G \times G$.
	\end{itemize}
\end{defn}

It is important to realize that most groups do not have a left-invariant total order that is recurrent:

\begin{lem}[{\cite[Cor.~4.4]{Morris-AmenOnLine}}] \label{RecurrentLO->LocIndic}
If $G$ has a left-invariant total order that is recurrent, then $G$ is locally indicable.
\end{lem}

\begin{proof}
Recall that a left-invariant total order~$\prec$ on~$G$ is said to be \emph{Conradian} \cite[Lem.~2.4.1(c)]{KopytovMedvedev-ROGrps} if for all $x,y \in G$ with $x , y \succ e$, there exists $n \in \N^+$, such that $x y^n \succ y$. It is easy to see that every recurrent left-invariant total order is Conradian (because there is some $n$ with $x y^n \succ y^n \succ y$), and it is well known that any group with a Conradian order must be locally indicable \cite{Conrad-ROGrps}, \cite[Thm.~2.4.1]{KopytovMedvedev-ROGrps}.
\end{proof}

The following theorem is the main result of \cite{Morris-AmenOnLine} (and is the culmination of a series of previous theorems of A.\,H.\,Rhemtulla, I.\,M.\,Chiswell, P.\,H.\,Kropholler, and P.\,A.\,Linnell that have stronger hypotheses in the place of ``amenable'').

\begin{thm}[{D.\,W.\,Morris {\cite{Morris-AmenOnLine}}}] \label{LOAmen}
If $G$ is a countable, amenable group, and $G$ has a left-invariant total order, then $G$ has a left-invariant total order that is recurrent.
\end{thm}

The proof actually establishes the following stronger statement:

\begin{prop} \label{ConvergeToRecur}
Let
\noprelistbreak
	\begin{itemize} \itemsep=\smallskipamount
	\item $G$ be a countable, amenable group,
	and
	\item $\rel$ be a binary relation on~$G$.
	\end{itemize}
Then there exists a sequence\/ $\{(g_n,h_n)\}_{n=1}^\infty$ of elements of $G \times G$, such that\/ $\{{\rel^{(g_n,h_n)}}\}_{n=1}^\infty$ converges to a binary relation that is recurrent.
\end{prop}

\begin{proof}
For the reader's convenience, we provide an outline of the proof.  See \cite{Morris-AmenOnLine} for more details of the main steps (\ref{AmenOnLOG}, \ref{PoincareLOG}, and~\ref{ReverseLOG}).
\noprelistbreak
	\begin{enumerate}  \itemsep=\smallskipamount
	\item \label{DefineClosure}
	Let $\closure{\rel^{G\times G}}$ be the closure of the ${G\times G}$-orbit of~$\rel$ in $2^{G\times G}$. 
	Note that $\closure{\rel^{G\times G}}$ is compact, since it is a closed subset of the compact space $2^{G\times G}$.
	\item \label{AmenOnLOG}
	Since ${G\times G}$ is amenable (see \fullcref{AmenFacts}{GxG}), there exists a $({G\times G})$-invariant probability measure on $\closure{\rel^{G\times G}}$.
	\item \label{PoincareLOG}
	Since there is an invariant probability measure, the Poincar\'e Recurrence Theorem \cite[Thm.~1]{Wikipedia-PoincareRecurrence} tells us, for each $(g,h) \in G \times G$, that almost every element of $\closure{\rel^{G\times G}}$ is recurrent for $(g,h)$.
	\item \label{ReverseLOG}
	Since $G \times G$ is countable, and the union of countably many sets of measure~$0$ is still a set of measure~$0$, we can reverse the quantifiers: for almost every $S \in \closure{\rel^{G\times G}}$, the binary relation~$S$ is recurrent for every $(g,h) \in {G\times G}$. 
	\item Since $G$ is countable, we know that $2^{G \times G}$ is a metric space, so there exists a sequence $\{(g_n,h_n)\}_{n=1}^\infty$ of elements of ${G\times G}$, such that ${\rel^{(g_n,h_n)}} \to {S}$ as $n \to \infty$.
	\qedhere
	\end{enumerate}
\end{proof}

\begin{cor} \label{RestrictAmenClosedSubset}
Let
\noprelistbreak
	\begin{itemize} \itemsep=\smallskipamount
	\item $G$ be a left-orderable group,
	\item $H$ be a countable, amenable subgroup of~$G$,
	and
	\item $\rels$ be a nonempty, closed, $(H \times H)$-invariant subset of\/ $2^{G \times G}$.
	\end{itemize}
Then there exists ${\rel} \in \rels$, such that the restriction of~$\rel$ to~$H$ is recurrent.
\end{cor}

\begin{proof}
Let ${\rel} \in \rels$, and let $\relH$ be the restriction of~$\rel$ to~$H$. Then \cref{ConvergeToRecur} provides a sequence\/ $\{(g_n,h_n)\}_{n=1}^\infty$ of elements of $H \times H$, such that\/ $\{{\relH^{(g_n,h_n)}}\}_{n=1}^\infty$ converges to a recurrent binary relation~$\relH^\infty$. 

Since $2^{G \times G}$ is compact, the sequence $\{{\rel^{(g_n,h_n)}}\}_{n=1}^\infty$ must have an accumulation point; call it~$\rel^\infty$. (Note that ${\rel^\infty} \in \rels$, since $\rels$ is closed and $(H \times H)$-invariant.)
Since the restriction map $2^{G \times G} \to2^{H \times H}$ is continuous, we know that the restriction of~$\rel^\infty$ to~$H$ must be an accumulation point of $\{{\relH^{(g_n,h_n)}}\}_{n=1}^\infty$. However, this sequence converges, so it has a unique accumulation point, namely~$\relH^\infty$. Therefore, the restriction of $\rel^\infty$ to~$H$ must be~$\relH^\infty$, which is recurrent.
\end{proof}

By letting $\rels$ be the set of left-invariant total orders, or the set of locally invariant orders, we see that the following two results are special cases of \cref{RestrictAmenClosedSubset}.

\begin{cor} \label{RestrictAmenIsRecur}
Let
\noprelistbreak
	\begin{itemize} \itemsep=\smallskipamount
	\item $G$ be a left-orderable group,
	and
	\item $H$ be a countable, amenable subgroup of~$G$.
	\end{itemize}
Then there exists a left-invariant total order~$\ll$ on~$G$, such that the restriction of~$\ll$ to~$H$ is recurrent.
\end{cor}

\begin{cor} \label{RestrictAmenIsRecurLIO}
Let
\noprelistbreak
	\begin{itemize} \itemsep=\smallskipamount
	\item $G$ be a group with a locally invariant order,
	and
	\item $H$ be a countable, amenable subgroup of~$G$.
	\end{itemize}
Then there exists a locally invariant order\/~$\prec$ on~$G$, such that the restriction of\/~$\prec$ to~$H$ is recurrent.
\end{cor}

\begin{rems} \ 
\noprelistbreak
	\begin{enumerate} \itemsep=\smallskipamount
	\item If $\prec$ is left-invariant, then $\prec^{(g,h)}$ is independent of~$g$, so we may write $\prec^h$.
	\item By letting $\rels = \closure{\prec^H}$, we see that the order~$\ll$ in the conclusion of \cref{RestrictAmenIsRecur} can be chosen to be in $\closure{\prec^H}$. 
	\item Furthermore, if $C$ is any countable subset of~$G$, then $\ll$ can be chosen so that $\ll$ is ``recurrent for~$H$ on~$C$.'' That is, for all $h \in H$ and all $x_1,x_2,\ldots,x_r \in C$ with $x_1 \prec x_2 \prec \cdots \prec x_r$, there exists $n \in \N^+$, such that $x_1 h^n \prec x_2 h^n \prec \cdots \prec x_r h^n$.
	\item Therefore, if $G$ is countable, then $\ll$ can be chosen to be recurrent for every element of~$H$, and there is a sequence $\{h_n\}_{n=1}^\infty$ of elements of~$H$, such that ${\prec^{h_n}} \to {\ll}$.
	\end{enumerate}
\end{rems}

\section{Proof of \cref{AMENLIO}} \label{AMENLIOPfSect}

\begin{prop} \label{DescribeRecurrent}
Let $\prec$ be a locally invariant order on~$G$ that is recurrent for all right-translations. Then:
	\begin{enumerate}
	\item \label{DescribeRecurrent-total}
	The restriction of $\prec$ to any left coset of any cyclic subgroup of~$G$ is either the standard linear order or its reverse. That is, for any $g,x \in G$, with $x \neq e$, we have either
		$$\cdots \prec gx^{-2} \prec gx^{-1} \prec g \prec gx \prec gx^2 \prec \cdots ,$$
	or
	$$ \cdots \succ gx^{-2} \succ gx^{-1} \succ g \succ gx \succ gx^2 \succ \cdots .$$
	{\upshape(}In particular, $\prec$ is a total order on~$G$.{\upshape)}
	\item \label{DescribeRecurrent-cone}
The positive cone of~$\prec$ is closed under multiplication.
	\item \label{DescribeRecurrent-LO}
	$G$ is left-orderable.
	\end{enumerate}
\end{prop}

\begin{proof}
\pref{DescribeRecurrent-total} Suppose this conclusion does not hold. Then, perhaps after replacing $g$ with~$g x^n$, for some $n \in \Z$, we have
	$$ \text{$g \prec gx \prec gx^2 \prec gx^3 \prec \cdots$ \ and \ $gx^{-1} \prec gx^{-2} \prec gx^{-3} \prec \cdots$} $$
Since $gx^{-1} \prec gx^{-2}$, and $\prec$ is recurrent for right-translation by~$x$, there exists $k \in\Z^+$, such that
	$(gx^{-1}) x^{k+2} \prec (gx^{-2}) x^{k+2}$.
This means $gx^{k+1} \prec gx^k$, which contradicts the fact that $g \prec gx \prec gx^2 \prec \cdots$

\pref{DescribeRecurrent-cone} Suppose there exist $x$ and~$y$, such that $x \succ e$ and $y \succ e$, but $xy \not\succ e$.  Since \pref{DescribeRecurrent-total} tells us that $\prec$ is a total order, we must have $xy \prec e$. Then $x \succ xy$ (because $x \succ e \succ xy$), so, from~\pref{DescribeRecurrent-total}, we must have
	$$ x \succ x y \succ x y^2 \succ \cdots ,$$
so 
	$$ \text{$e \succ xy \succeq xy^n$, for all $n \in \mathbb{Z}^+$} .$$
On the other hand, since $\prec$ is recurrent for right-translation by~$y$, and $x \succ e$, we know there is some $n \in \mathbb{Z}^+$, such that $x y^n \succ e y^n \succ e$. This is a contradiction.

\pref{DescribeRecurrent-LO} Let $P = \{\, x \in G \mid x \succ e \,\}$ be the positive cone of~$\prec$. For any~$x \in G$ with $x \neq e$, letting $g = e$ in \pref{DescribeRecurrent-total} tells us that either $x \in P$ or $x^{-1} \in P$ (but not both). Furthermore, \pref{DescribeRecurrent-cone} tells us that $P$ is closed under multiplication. Therefore $P$ is the positive cone of a left-invariant total order on~$G$ \cite[Thm.~1.5.1]{KopytovMedvedev-ROGrps} (but the left-invariant order may be different from~$\prec$).
\end{proof}

\begin{proof}[\bf Proof of \cref{AMENLIO}]
Assume $G$ is an amenable group that has a locally invariant order. We wish to show that $G$ is left-orderable. There is no harm in assuming that $G$ is finitely generated \cite[Cor.~3.1.1]{KopytovMedvedev-ROGrps}, and hence countable. Then \cref{RestrictAmenIsRecurLIO} (with $H = G$) tells us that $G$ has a locally invariant order that is recurrent. So \fullcref{DescribeRecurrent}{LO} tells us that $G$ is left-orderable.

Now \cref{LOAmen,RecurrentLO->LocIndic} tell us that $G$ is locally indicable.
\end{proof}

\section{Proof of \cref{RESTONILPISBI}} \label{RESTONILPISBIPfSect}

\begin{notation}
Let $x$ and~$h$ be elements of a group~$H$.
	\begin{itemize}
	\item We use $x^h$ to denote the {conjugate} $h^{-1} x h$.
	\item We use $[x,h]$ to denote the {commutator} $x^{-1} h^{-1}x h = x^{-1} x^h$.
	\end{itemize}
\end{notation}

\begin{lem} \label{NilpRecur->Bi}
If\/ $\prec$ is a recurrent left-invariant total order on a locally nilpotent group~$H$, then\/ $\prec$ is bi-invariant.
\end{lem}

\begin{proof}
Let $P = \{\, x \in H \mid x \succ e \,\}$ be the positive cone of~$\prec$. We wish to show $P$ is invariant under conjugation by elements of~$H$.

Arguing by contradiction, let us assume there exist $x,h \in H$, such that
	$$ \text{$x \succ e$ \ and \ $x^h \prec e$.} $$
Since $\{x,h\}$ is finite, there is no harm in assuming $H$ is finitely generated. Hence, $H$ is nilpotent, so there is a central series
	$$H = H_r \triangleright H_{r-1} \triangleright \cdots \triangleright H_1 \triangleright H_0 = \{e\} ,$$
such that $[H_k, H] \subset H_{k-1}$ for every~$k$.

Fix $k$, such that $x \in H_k$, and assume,  by induction, that $P \cap H_{k-1}$ is invariant under conjugation by elements of~$H$. Since $x \succ e$, but $x [x,h] = x^h \prec e$, we must have $[x,h] \prec e$. Then, since
	$ [x,h] \in [H_k, H] \subset H_{k-1}$,
our induction hypothesis tells us that 
	$$ \text{$[x,h]^{h^i} \prec e$ for every $i \in \Z$.} $$
Therefore, for every $n \in \Z^+$, we have
	\begin{align*}
	 x^{h^n} 
	 &= x^h \, [x^h,h] \,  [x^{h^2},h] \, \cdots \, [x^{h^{n-1}},h]
	 \\&= x^h \, [x,h]^h \,  [x,h]^{h^2} \cdots \,  [x,h]^{h^{n-1}}
	 \\&\prec e
	 . \end{align*}
Since $x \succ e$, this contradicts the fact that $\prec$ is recurrent.
\end{proof}

\begin{proof}[{\bf Proof of \cref{RESTONILPISBI}}]
Assume, for the moment, that $H$ is countable. Then, since $H$ is amenable (see \cref{LocNilpIsAmen}), \Cref{RestrictAmenIsRecur} provides us with a left-invariant total order~$\ll$ on~$G$, such that the restriction of~$\ll$ to~$H$ is recurrent. \cref{NilpRecur->Bi} tells us that the restriction to~$H$ must be bi-invariant, as desired.

Now consider the general case. 
	\begin{itemize}
	\item Let $\LO(G)$ be the set of left-invariant total orders on~$G$.
	\item For each subgroup~$K$ of~$H$, let
	$$ B_G(K) = \{\, {\prec} \in \LO(G) \mid \text{the restriction of~$\prec$ to~$K$ is bi-invariant} \,\} .$$
	\item Let $\mathcal{C}$ be the collection of countable subgroups of~$H$. 
	\end{itemize}
For $K_1,\ldots,K_n \in \mathcal{C}$, the subgroup $\langle K_1,\ldots,K_n \rangle$ is countable, so the first paragraph of the proof implies that 
	$$B_G(K_1) \cap \cdots \cap B_G(K_n) \  \supset  \ B_G \bigl( \langle K_1,\ldots,K_n \rangle \bigr) \ \neq \ \emptyset  .$$
Since each $B_G(K)$ is easily seen to be a closed subset of $\LO(G)$, and $\LO(G)$ is compact, we conclude that $\bigcap_{K \in \mathcal{C}} B_G(K) \neq \emptyset$. Since every finite subset of~$H$ is contained in an element of~$K$, we know that any element of this intersection is a left-invariant total order on~$G$ whose restriction to~$H$ is bi-invariant, as desired.
\end{proof}

\begin{rems} \label{Generalize} \ 
\noprelistbreak
	\begin{enumerate} \itemsep=\smallskipamount
	\item \label{Generalize-posicyclic}
The bi-invariance of all recurrent orders holds for a more general class of amenable groups than just those that are locally nilpotent. For example, let us say that~$G$ is \emph{positively polycyclic} if $G$ is a polycyclic group that is isomorphic to a group of upper-triangular $n \times n$ real matrices with all diagonal entries positive (for some~$n$). Generalizing \cref{NilpRecur->Bi}, it can be shown that if $G$ is a locally positively polycyclic group, then every recurrent left-invariant total order on~$G$ is bi-invariant. Therefore, \cref{RESTONILPISBI} remains valid if the word ``nilpotent'' is replaced with ``positively polycyclic''.
	\item \label{Generalize-NotForResid}
On the other hand, the word ``locally'' in \cref{RESTONILPISBI} cannot be replaced with the phrase ``residually torsion-free,'' even if we add the additional assumption that $H$ has finite index in~$G$. 
For example, a braid group on $5$ or more strands has no left-order whose restriction to a subgroup of finite index is bi-invariant \cite{DubrovinaDubrovin-BraidGrps}, 
\cite[Thm.~3.2]{RhemtullaRolfsen-LocIndicBraids}, even though the subgroup of pure braids is a subgroup of finite index that is residually torsion-free nilpotent.
	\end{enumerate}
\end{rems}

\section{Diffuse groups and weakly diffuse groups} \label{DiffuseSect}

\begin{defn}[\cite{Bowditch-VarUPP}]
Let $G$ be a group.
	\begin{enumerate}
	\item An element~$\widehat a$ of a subset~$A$ of~$G$ is an \emph{extreme point} of~$A$ if, for all nonidentity $h \in G$, we have either $\widehat a \, h \notin A$ or $\widehat a \, h^{-1} \notin A$. Equivalently, we have $\widehat a^{-1} A \cap A^{-1} \widehat a = \{e\}$, where $A^{-1} = \{\, a^{-1} \mid a \in A \,\}$.
	\item $G$ is \emph{weakly diffuse} if every nonempty, finite subset of~$G$ has an extreme point.
	\item $G$ is \emph{diffuse} if every finite subset~$A$ of~$G$ with $\#A \ge 2$ has at least two extreme points.
	\end{enumerate}
\end{defn}

Answering questions of B.\,Bowditch \cite[p.~815]{Bowditch-VarUPP} and I.\,Chiswell \cite[p.~1163]{Chiswell-LIO}, we observe that the above two properties of~$G$ are equivalent to the existence of a locally invariant order relation:

\begin{prop} \label{DiffuseIff}
For any group~$G$, the following are equivalent:
	\begin{enumerate}
	\item \label{DiffuseIff-diffuse}
	$G$ is diffuse.
	\item \label{DiffuseIff-weak}
	$G$ is weakly diffuse.
	\item \label{DiffuseIff-total}
	$G$ has a locally invariant total order.
	\item \label{DiffuseIff-partial}
	$G$ has a locally invariant partial order.
	\end{enumerate}
\end{prop}

\begin{proof}
We prove $(\ref{DiffuseIff-diffuse}) \Leftrightarrow (\ref{DiffuseIff-weak})$ and
$(\ref{DiffuseIff-weak}) \Rightarrow (\ref{DiffuseIff-total}) \Rightarrow (\ref{DiffuseIff-partial}) \Rightarrow (\ref{DiffuseIff-weak})$.
Begin by noting that $(\ref{DiffuseIff-diffuse}) \Rightarrow (\ref{DiffuseIff-weak})$ and
$(\ref{DiffuseIff-total}) \Rightarrow (\ref{DiffuseIff-partial})$ are trivial. Also, $(\ref{DiffuseIff-partial}) \Rightarrow (\ref{DiffuseIff-weak})$ is well known (and not difficult) \cite[Lem.~1.2(2)]{Chiswell-LIO}.

$(\ref{DiffuseIff-weak} \Rightarrow \ref{DiffuseIff-diffuse})$
Suppose $G$ is weakly diffuse, but not diffuse. Then there is a finite subset~$A$ of~$G$ with $\#A \ge 2$, such that the extreme point of~$A$ is unique. After multiplying $A$ on the left by an element of~$G$, we may assume that the extreme point of~$A$ is~$e$. Then $e$ is also the unique extreme point of~$A^{-1}$. (In general, since 
$\widehat a^{-1} A \cap A^{-1} \widehat a = \widehat a^{-1} \bigl(
\widehat a A^{-1} \cap A \widehat a^{-1} \bigr) \widehat a$,
we see that $\widehat a$ is an extreme point of~$A$ if and only
$\widehat a^{-1}$ is an extreme point of~$A^{-1}$.) So the only
possible extreme point of $A \cup A^{-1}$ is~$e$. However, if we let
$h$ be any nonidentity element of~$A$, then we have $eh^{\pm1} =
h^{\pm1} \in A \cup A^{-1}$, so $e$ is not an extreme point. Therefore $A \cup A^{-1}$ has no extreme points, which contradicts the fact that $G$ is weakly diffuse.

$(\ref{DiffuseIff-weak} \Rightarrow \ref{DiffuseIff-total})$
Assume $G$ is weakly diffuse. We wish to show that $G$ has a locally invariant total order.
By a straightforward compactness argument, it suffices to show that every finite subset~$A$ of~$G$ has a total order $<_A$ with the following property:
	\begin{align} \label{finite}
	\begin{matrix}
	\text{for all $a \in A$ and all nonidentity $h \in G$, 
		such that $ah \in A$ and $ah^{-1} \in A$,} \\
	 \text{we have either $a <_A ah$ or $a <_A ah^{-1}$} 
	 . \end{matrix}
	\end{align}
We construct~$<_A$ by induction on the cardinality of~$A$.  Since $G$ is weakly diffuse, there exists an extreme point $\widehat a$ of~$A$. (Note that the condition in~(\ref{finite}) is vacuously true for $a = \widehat a$.) By the induction hypothesis, there is a total order on $A \smallsetminus \{\widehat a\}$ that satisfies (\ref{finite}) when $A$ is replaced with $A \smallsetminus \{\widehat a\}$. We extend this to a total order on~$A$ by specifying that $\widehat a$ is the unique maximal element. Then the resulting order satisfies~(\ref{finite}).
\end{proof}

Thus, \cref{AMENLIO} can be restated in the following form:

\begin{thm}
An amenable group is weakly diffuse if and only if it is locally indicable.
\end{thm}

\end{document}